%% file: biSobolev.tex
\documentclass[11pt,a4paper,reqno]{amsart}

\usepackage{amsmath,amsthm,amsfonts,graphicx}

\newtheorem{theorem}{Theorem}[section]
\newtheorem{lemma}[theorem]{Lemma}
\newtheorem{corollary}[theorem]{Corollary}
\newtheorem{prop}[theorem]{Proposition}
\newtheorem{remark}[theorem]{Remark}

\newcounter{mt}
\def\maintheorem#1#2#3{\par \medskip \noindent {\bf Theorem~\mref{#1}}~(#2).~{\it #3}\par}
\def\mref#1{\Alph{#1}}
\def\maintheoremdeclaration#1{\stepcounter{mt}\newcounter{#1}\setcounter{#1}{\arabic{mt}}}
\maintheoremdeclaration{Approxp=1}
\def\eps{\varepsilon}

\def\R{\mathbb R}

\title[On the bi-Sobolev planar homeomorphisms and their approximation]
{On the bi-Sobolev planar homeomorphisms\\and their approximation}
\author{Aldo Pratelli}
\address{Department of Mathematics, University of Erlangen, Cauerstrasse 11, 90158 Erlangen, Germany}
\email{\tt pratelli@math.fau.de}
\begin{document}

\begin{abstract}
The first goal of this paper is to give a short description of the planar bi-Sobolev homeomorphisms, providing simple and self-contained proofs for some already known properties. In particular, for any such homeomorphism $u:\Omega\to \Delta$, one has $Du(x)=0$ for almost every point $x$ for which $J_u(x)=0$. As a consequence, one can prove that
\begin{equation}\tag{\ref{1.1}}
\int_\Omega |Du| = \int_\Delta |Du^{-1}|\,.
\end{equation}
Notice that this estimate holds trivially if one is allowed to use the change of variables formula, but this is not always the case for a bi-Sobolev homeomorphism.\par
As a corollary of our construction, we will show that any $W^{1,1}$ homeomorphism $u$ with $W^{1,1}$ inverse can be approximated with smooth diffeomorphisms (or piecewise affine homeomorphisms) $u_n$ in such a way that $u_n$ converges to $u$ in $W^{1,1}$ and, at the same time, $u_n^{-1}$ converges to $u^{-1}$ in $W^{1,1}$. This positively answers an open conjecture (see for instance~\cite[Question~4]{IwKovOnn}) for the case $p=1$.
\end{abstract}

\maketitle

\section{Introduction}

In this paper, we are interested in bi-Sobolev planar homeomorphisms, that are homeomorphisms $u\in W^{1,1}(\Omega)$ such that $u^{-1}\in W^{1,1}(\Delta)$, where we denote for brevity $\Delta=u(\Omega)$. Since for any $2\times 2$ invertible matrix $M$ one has the trivial equality $|M| = |\det M| |M^{-1}|$, if the change of variables formula holds then we can formally calculate
\[
\int_\Delta |Du^{-1}(y)|\, dy = \int_\Omega |Du^{-1}(u(x))| |J_u(x)|\, dx =
\int_\Omega |Du(x)^{-1}| |J_u(x)|\, dx = \int_\Omega |Du(x)|\, dx\,.
\]
There are two problems in the above calculation. First of all, it is well known that the change of variables formula is true only if the ``Lusin $N$ property'' holds for $u$, which means that $u$ maps null sets in null sets; otherwise, in general only an inequality holds (see for instance~\cite[Theorem~A.13]{HKLectureNotes}, or Lemma~\ref{cov} below). And several examples show that the Lusin $N$ property may fail for bi-Sobolev homeomorphisms; some of these examples can be found for instance in~\cite{Pon,DHS,HKLectureNotes}. A second problem is that the formal calculation above requires the existence of $Du^{-1}$ at the point $u(x)$; since $u^{-1}$ is in $W^{1,1}$, then we know that $Du^{-1}$ exists at almost every point $y$, however the points $x$ such that $u(x)$ belongs to a negligible set in $\Delta$ need not to be negligible in $\Omega$, even for bi-Sobolev homeomorphisms. Requiring that $u^{-1}$ maps negligible sets of $\Delta$ in negligible sets of $\Omega$ is the so-called ``Lusin $N^{-1}$ property''. Summarizing, the validity of formula~(\ref{1.1}) can be immediately obtained with the calculations above only if $u$ satisfies both the $N$ and the $N^{-1}$ Lusin properties. The general validity of the formula for bi-Sobolev homeomorphisms is still true, but the proof is more involved. One of the main ingredients to show it, is an interesting property of bi-Sobolev homeomorphisms, namely, that for almost every $x$ for which $J_u(x)=0$, one has actually $Du(x)=0$. This technical fact, which might seem surprising at first sight, plays a key role in the study of bi-Sobolev homeomorphisms; roughly speaking, the main reason is that often the points for which $J_u$ is zero while $Du$ is non-zero give heavy troubles in the constructions. In the last years, an extremely fine work on bi-Sobolev maps has been done by several authors, among which Sbordone and Hencl. Together with their coauthors, they could prove all the main properties of these maps, and in particular the two described above.\par

In the last period, bi-Sobolev maps have then become of big interest also because of their important role in the non-linear elasticity, which we describe quickly here (for more precise and better explanations, one can see for instance~\cite{Ball2,IwKovOnnold,IwKovOnn,HP}). In the framework of nonlinear elasticity, it is extremely important to approximate a Sobolev (or bi-Sobolev) homeomorphism with piecewise affine homeomorphisms, or diffeomorphisms. By ``piecewise affine homeomorphism'' we always mean a homeomorphism which is affine on each triangle of a locally finite triangulation of $\Omega$; one can also try to get a globally finite triangulation whenever this makes sense (for instance, this is of course impossible if $\Omega$ is not a polygon!). The result which is needed is basically the following one.\\
{\bf Question~1.} Let $u\in W^{1,p}(\Omega)$ be a homeomorphism, for $p\geq 1$. Is it true that there exists a sequence $u_n$ of piecewise affine homeomorphisms, or of diffeomorphisms, such that $u_n$ converges to $u$ in $W^{1,p}$?\\
{\bf Question~2.} Let $u\in W^{1,p}(\Omega)$ be a homeomorphism such that $u^{-1}\in W^{1,p}(u(\Omega))$, for $p\geq 1$. Is it true that there exists a sequence $u_n$ of piecewise affine homeomorphisms, or of diffeomorphisms, such that $u_n$ converges to $u$ in $W^{1,p}$ and, at the same time, $u_n^{-1}$ converges to $u^{-1}$ in $W^{1,p}$?\\
The above questions are not new at all, as far as we know they were first set as open questions by J. Ball and L. Evans. For instance, in~\cite{Ball2} Question~1 for the three-dimensional case was asked. In the last years, many attempts to answer to these questions have been done, all for the two-dimensional case. In particular, few years ago Iwaniec, Kovalev and Onninen positively answered to Question~1 in the case when $p>1$ in~\cite{IwKovOnn} (see also~\cite{IwKovOnnold}), while the case $p=1$ of Question~1 was very recently solved in~\cite{HP}. Unfortunately, for a broad application in the framework of nonlinear elasticity a positive answer to Question~1 is not enough, because the $W^{1,p}$ convergence of $u_n$ to $u$ is not even enough to ensure that the elastic energy of $u_n$ is similar to the one of $u$. Indeed, depending on the applications, the elastic energy is always something of the form $\int_\Omega W(Du)$, where the functional $W$ explodes when $\det(Du)$ is very small (see for instance~\cite{Ball2}). This is basically the reason why Question~2 naturally arises: for instance, this was set as an open question in the paper~\cite{IwKovOnn}. Up to now, the only partial answer to Question~2 was found in~\cite{DP}: there, the authors give the affirmative answer, but only under the very strong assumption that $u$ is a bi-Lipschitz function. The main result of this paper is the positive answer to Question~2 for the case $p=1$.

\maintheorem{Approxp=1}{Approximation of bi-Sobolev homeomorphisms with $p=1$.}{Let $\Omega\subseteq \R^2$ be an open set, and let $u\in W^{1,1}(\Omega)$ be a homeomorphism such that $u^{-1}\in W^{1,1}(u(\Omega))$. Then, for every $\eta>0$ there exists a diffeomorphism, as well as a piecewise affine homeomorphism, $u_\eta$ on $\Omega$, with the property that
\begin{equation}\label{estimate}
\|u-u_\eta\|_{L^\infty}+\|u^{-1}-u_\eta^{-1}\|_{L^\infty} + \|Du-Du_\eta\|_{L^1}+\|Du^{-1}-Du_\eta^{-1}\|_{L^1} \leq \eta\,.
\end{equation}
Moreover, we have that $u_\eta(\Omega)=u(\Omega)$ and, as soon as $u$ is continuous up to the boundary, then so is $u_\eta$ and $u_\eta=u$ on $\partial\Omega$. Finally, the triangulation corresponding to the piecewise affine $u_\eta$ can be chosen to be finite, instead of locally finite, as soon as $\Omega$ is a polygon and $u$ is piecewise linear on $\partial \Omega$.}\bigskip

In view of this application in the nonlinear elasticity, many of the important general properties of bi-Sobolev maps are then useful, but they are often regarded from a quite different point of view. In particular, some results which are extremely important for this application are somehow less important in the general study of bi-Sobolev maps, and vice versa. For instance, the validity of~(\ref{1.1}), which is fundamental to obtain Theorem~\mref{Approxp=1}, can be found in the literature only without a proof, since it can be deduced from other results.\par

Because of this different view on the matter, we have decided to give here a short and self-contained description of those aspects of bi-Sobolev homeomorphisms which are important for our construction. We will also give a proof of the main results that we will need to use. In fact, even if they are already known, as said above they are a bit marginal in the biggest framework of the bi-Sobolev maps, so for the interested reader it is simpler to find all what he needs in few pages here, instead of having to check in several different papers. Moreover, our proofs have a quite different flavour than the ``classical'' ones, and they contain the same basic ideas that will be then needed in the proof of our main result.\par

The plan of the paper is very simple. First of all, in Section~\ref{prelim} we will present some preliminaries, mainly known results from recent papers which will be used in our construction. Then, Section~\ref{A} is devoted to show the main properties of bi-Sobolev homeomorphisms which we will need later, and Section~\ref{B} contains the proof of Theorem~\mref{Approxp=1}.

\section{Preliminaries and known facts\label{prelim}}

In this section we present some known facts about bi-Sobolev homeomorphisms, and about the general approximation problem. Through the paper, $u$ will always denote a bi-Sobolev homeomorphism between two open sets $\Omega,\, \Delta\subseteq \R^2$. For any point $x\in\R^2$, we will denote by $Q_r(x)$ the square of side $r$ centered at $x$, with sides parallel to the coordinate axes, and as usual $J_u(x)=\det(Du(x))$ will be the Jacobian of $u$ at $x$. First of all, let us present the general version of the change of variables formula, for the special case of homeomorphisms: this is now a well-established result, see for instance~\cite[Theorem~A.13]{HKLectureNotes}
\begin{lemma}[Change of variables formula]\label{cov}
Let $f\in W^{1,1}_{{\rm loc}}(\Omega,\Delta)$ be a homeomorphism, and let $\varphi:\Omega\to\R^+$ be a measurable function. Then the inequality
\begin{equation}\label{changeofvariables}
\int_\Omega \varphi(f(x)) |J_f(x)|\, dx \leq \int_\Delta \varphi(y)\,dy
\end{equation}
holds, with equality if $f$ satisfies the Lusin $N$ property.
\end{lemma}
Thanks to this formula, a simple calculation provides a quantitative estimate of how much $u$ is similar to its first-order expansion near a Lebesgue point for $Du$.
\begin{lemma}\label{oldsta}
Let $x\in\Omega$ be a Lebesgue point for $Du$. Then, for every $\eps>0$ there exists $\bar r_1=\bar r_1(x)>0$ such that, for every $r<\bar r_1$ and every $\tilde x \in Q_r(x)$, the following estimates hold, where $v(z)= u(x) + Du(x) (z-x)$ is the first-order expansion of $u$ near $x$,
\begin{align}\label{stimagen}
\| u - v \|_{L^\infty(Q_{3r}(\tilde x))} \leq \eps r\,, && \int_{Q_{3r}(\tilde x)} |Du(z) - Du(x) | \, dz\leq \eps r^2\,.
\end{align}
Moreover, for almost every $x$ such that $J_u(x)\neq 0$, we have also some $\bar r_2=\bar r_2(x)>0$ such that $\bar r_2(x)\leq \bar r_1(x)$ and, for every $r<\bar r_2$ and again every $\tilde x\in Q_r(x)$,
\begin{equation}\label{stimapar}
\int_{u(Q_{3r}(\tilde x))} |Du^{-1}(w) - Du^{-1}(u(x)) | \, dw \leq \eps r^2\,.
\end{equation}
\end{lemma}
\begin{proof}
First of all, observe that it is enough to show the claim for the particular case $\tilde x=x$, since $\tilde x\in Q_r(x)$ implies $Q_r(\tilde x)\subseteq Q_{2r}(x)$.\par

The estimates of~(\ref{stimagen}) are well known, for instance they are proved in~\cite[Lemma~4.2]{DP}. Basically, the second estimate comes directly from the definition of Lebesgue points, and a very simple argument via Fubini Theorem gives then also the first one.\par

Let us then concentrate ourselves on~(\ref{stimapar}). Assume that $x$ is a Lebesgue point for $Du$, and that $J_u(x)\neq 0$; as a consequence, if we call $y=u(x)$ then we have that $Du^{-1}(y)$ exists and coincides with the inverse of $Du(x)$. Suppose for a moment that $y$ is in fact a Lebesgue point for $Du^{-1}$: in this case, the estimate~(\ref{stimapar}) comes again directly from the definition of Lebesgue points, also recalling that $Du^{-1}(y)$ is invertible and using the $L^\infty$ estimate of~(\ref{stimagen}). Therefore, to conclude we only have to show that for almost every $x$ with $J_u(x)\neq 0$ the point $u(x)$ is a Lebesgue point for $Du^{-1}$.\par
To do so, let us define $\Gamma\subseteq \Delta$ the set of points of $\Delta$ which are not Lebesgue points for $Du^{-1}$; keep in mind that we conclude the proof once we show that $J_u(x)=0$ for almost every $x\in u^{-1}(\Gamma)$. We know that $|\Gamma|=0$ since $u^{-1}$ belongs to $W^{1,1}(\Delta)$. We can then apply the change of variables formula~(\ref{changeofvariables}) with $f=u$ and $\varphi=\chi_\Gamma$, and conclude the proof since
\[
\int_{u^{-1}(\Gamma)} |J_u(x)|\, dx =
\int_{\Omega} \varphi(u(x)) |J_u(x)|\, dx
\leq \int_\Delta \varphi(y)\, dy = |\Gamma|=0\,.
\]
\end{proof}

We can observe the following corollary of the above estimate, valid for the case when $J_u(x)\neq 0$. In this case, the $L^\infty$ estimate~(\ref{stimagen}) implies that the images under $u$ of the four vertices of the square $Q=Q_r(\tilde x)$ are very close to the vertices of the parallelogram $v(Q)$. We can then define, on the square $Q$, the ``interpolation'' $u_Q$ as the function which coincides with $u$ on the four vertices of $Q$, and which is affine on the two triangles in which $Q$ is subdivided by a diagonal (it doesn't matter which of the diagonals we choose). A trivial calculation, starting from~(\ref{stimagen}) and~(\ref{stimapar}), gives then the following estimate (almost identical calculations can be found in~\cite{HP,DP}).
\begin{corollary}\label{bafo}
Let $x\in\Omega$ be a Lebesgue point for $Du$ such that $J_u(x)\neq 0$ and the second part of Lemma~\ref{oldsta} holds true for $x$. If $Q$ is a square of side $r$ containing $x$, with $r<\bar r_2(x)$, then the interpolated function $u_Q$ satisfies
\[
\int_Q |Du-Du_Q| + \int_{u_Q(Q)} |Du^{-1} - Du_Q^{-1}| < 5 \eps r^2\,.
\]
\end{corollary}

We can now briefly explain the overall strategy which was used in~\cite{HP} to solve Question~1 in the case $p=1$ (which is actually similar to the strategy previously used in~\cite{DP}): we will use an analogous strategy also in here. For any $r>0$, consider the squares of side $r$ whose centers have both the coordinates integer multiples of $r$: these are pairwise disjoint squares which cover $\R^2$. In particular, we will call ``\emph{$r$-tiling of $\Omega$}'' the set of those squares $Q_r(x)$ such that $Q_{3r}(x)$ is compactly contained in $\Omega$. Assuming for simplicity that $\Omega$ has finite area, the squares of the $r$-tiling are finitely many, and they cover the whole $\Omega$ up to a set whose measure goes to $0$ when $r$ goes to $0$. The idea of the proof is then to consider separately the squares for which the above results apply, which are ``easier'' to treat, and the squares for which this does not hold true, which are ``worse'' but whose total area is arbitrarily small. The case when $\Omega$ has infinite area does not give any serious additional trouble, since it is possible to regard it as the countable union of increasing open sets of finite area, and all the results can be then more or less automatically extended from the finite area case to the general one. We will need to use the following result, which is taken from~\cite[Theorem~2.1 and Section~4]{HP}: in fact, a much more complicated result is proved there, but we claim here only the particular case that we are going to need.
\begin{prop}\label{standa}
Let $\Omega\subseteq\R^2$ be an open set, and let $u\in W^{1,1}(\Omega;\R^2)$ be a homeomorphism. Let $\{ Q_i\}$, for $1\leq i \leq N$, be a finite union of squares of the $r$-tiling of $\Omega$, such that any square $Q_i$ contains at least a Lebesgue point $x$ for $Du$ such that $J_u(x)\neq 0$ and $\bar r_1(x) > r$. Then, there exists a piecewise affine homeomorphism $v\in W^{1,1}(\Omega; \R^2)$ such that:
\begin{enumerate}
\item[(i)] $\|u-v\|_{L^\infty} + \|u^{-1}-v^{-1}\|_{L^\infty} < \eps$;
\item[(ii)] $v(\Omega)=u(\Omega)$ and, if $u$ is continuous up to $\partial\Omega$, then so is $v$ and $v=u$ on $\partial\Omega$;
\item[(iii)] the triangulation of $v$ is finite if $\Omega$ is a polygon and $u$ is piecewise linear on $\partial\Omega$;
\item[(iv)] $v$ coincides with the interpolation $u_{Q_i}$ for every $1\leq i \leq N$;
\item[(v)] the following estimate holds, being $K$ a purely geometric constant,
\[
\int_{\Omega\setminus \cup_{i=1}^N Q_i} |D v| \leq K \int_{\Omega\setminus \cup_{i=1}^N Q_i} |D u| \,.
\]
\end{enumerate}
\end{prop}

We conclude this section with a couple of comments about the last result. First of all, notice carefully that the function $v$ coincides with the interpolation $u_Q$ only for \emph{some} of the squares $Q$ of the $r$-tiling, not for all them: when using the above proposition, the ``right'' choice of which squares of the tiling have to be taken is also important. Second, the constant $\bar r_1(x)$ mentioned in the proposition is the same as in Lemma~\ref{oldsta}: in fact, the assumption $u^{-1}\in W^{1,1}$ is needed there only to get the constant $\bar r_2(x)$, while the constant $\bar r_1(x)$ only needs $u\in W^{1,1}$. Finally, let us comment about the possibility of having a finite triangulation, instead of just a locally finite one: this is of course nicer but, since we want the function $v$ to be defined in $\Omega$ and to coincide with $u$ on $\partial\Omega$, this is also impossible unless $\Omega$ is a polygon and $u$ is already piecewise linear on $\partial\Omega$. As a consequence, the additional assumptions in point~(iii) of the above proposition are sharp.

\section{Properties of bi-Sobolev homeomorphisms\label{A}}

In this section we analyse some properties of bi-Sobolev homeomorphisms, and in particular we prove the validity of~(\ref{1.1}). As said in the introduction, most of the results of this section are already known, but we prefer to give also their proofs, because they are simple and the constructions are similar to those that we will need to get Theorem~\mref{Approxp=1}; moreover, in this way this paper remains self-contained. Let then $u:\Omega\to \Delta$ be a bi-Sobolev homeomorphism. We start considering the situation around a Lebesgue point $x$ for $Du$ such that $J_u(x)=0$ but $Du(x)\neq 0$.

\begin{lemma}\label{lucio}
Let $x\in\Omega$ be a Lebesgue point for $Du$ such that $J_u(x)=0 < |Du(x)|$. Then, for any $\eps>0$ there exists $\bar r_3=\bar r_3(x)>0$ such that, for any $r<\bar r_3$ and any $\tilde x\in Q_r(x)$,
\begin{align}\label{flightrome}
\big|u(Q_r(\tilde x))\big| < \eps \int_{u(Q_r(\tilde x))} |Du^{-1}|\,, && 
\int_{u(Q_r(\tilde x))} |Du^{-1}| \geq (1-\eps) \int_{Q_r(\tilde x)} |Du|\,.
\end{align}
\end{lemma}
\begin{proof}
Up to a rotation, we can assume that
\[
Du(x) = \left(\begin{matrix} L & 0\\ 0 & 0 \end{matrix}\right)
\]
for some $L>0$. Let now $\delta= \delta(\eps,L)$ be a small constant, to be specified later, and let us apply the first part of Lemma~\ref{oldsta} with the constant $\delta$: we find a constant $\bar r_3=\bar r_3(x)$ such that for any $r<\bar r_3$ the uniform estimate
\begin{equation}\label{infest}
\| u -v \|_{L^\infty(Q_r(\tilde x))} \leq \delta r
\end{equation}
holds. In particular, $u(Q_r(\tilde x))$ is very close to $v(Q_r(\tilde x))$, which is an horizontal segment of length $Lr$; to fix the ideas, assume $v(Q_r(\tilde x))= \big\{w\equiv (w_1,0)\in\R^2:\, 0\leq w_1 \leq Lr \big\}$. Figure~\ref{Fig:Jf=0} depicts the situation.\par
\begin{figure}[thbp]
\input{Jf=0.tex}
\caption{The set $u(Q_r(x))$ in Lemma~\ref{lucio}.}\label{Fig:Jf=0}
\end{figure}
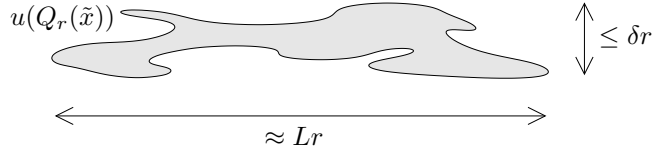
Let us now fix $\delta r < t < (L-\delta) r$, and let $w=(t,w_2)$ be a point in the boundary of $u(Q_r(\tilde x))$. Since $u$ is a homeomorphism, this means that $w=u(z)$ for some $z\in \partial Q_r(\tilde x)$, and the $L^\infty$ estimate~(\ref{infest}) ensures that $z$ cannot belong to the left nor the right side of $\partial Q_r(\tilde x)$. We can say even something more precise: since the homeomorphism $u$ transforms the closed curve $\partial Q_r(\tilde x)$ in the closed curve $u(\partial Q_r(\tilde x))$, the boundary of $u(Q_r(\tilde x))$ must contain at least some point $w=(t,w_2)=u(z)$ such that $z$ belongs to the upper side of $\partial Q_r(\tilde x)$, and also some other point $w'=(t,w_2')=u(z')$ such that $z'$ belongs to the lower side of $\partial Q_r(\tilde x)$. As an immediate consequence, we get
\[
\int_{u(Q_r(\tilde x)) \cap \{ w:\, w_1=t\}} |Du^{-1}(w)|\, dw \geq \int_{u(Q_r(\tilde x)) \cap \{ w:\, w_1=t\}} |D_2 u^{-1}(w)|\, dw \geq r.
\]
Since this estimate holds for every $\delta r < t <(L-\delta) r$, we deduce
\[
\int_{u(Q_r(\tilde x))} |Du^{-1}| \geq (L-2\delta) r^2\,.
\]
On the other hand, from the estimate~(\ref{stimagen}) we infer that
\[
\int_{Q_r(\tilde x)} |Du(z)| \, dz \leq 
L r^2 + \int_{Q_r(\tilde x)} |Du(z) - Du(x) | \, dz
\leq (L+\delta) r^2\,.
\]
Finally, the $L^\infty$ estimate~(\ref{infest}) immediately ensures that
\[
\big| u(Q_r(\tilde x))\big| \leq (L+2\delta)\delta r^2\,.
\]
Putting together the last three estimates immediately yields the validity of~(\ref{flightrome}), as soon as $\delta$ is small enough depending only on $\eps$ and on $L$, thus ultimately only on $\eps$ and $x$.
\end{proof}

We can immediately observe that the second estimate in~(\ref{flightrome}) holds true also for points where $J_u(x)\neq 0$: this comes as an immediate corollary of Lemma~\ref{oldsta}.
\begin{lemma}\label{lucio2}
For almost every $x\in \Omega$ such that $J_u(x)\neq 0$, and for any $\eps>0$, there exists $\bar r_3=\bar r_3(x)>0$ with $\bar r_3(x)\leq \bar r_2(x)$ and such that, for any $r<\bar r_3$ and every $\tilde x\in Q_r(x)$,
\begin{equation}\label{flight2}
\int_{u(Q_r(\tilde x))} |Du^{-1}| \geq (1-\eps) \int_{Q_r(\tilde x)} |Du|\,.
\end{equation}
\end{lemma}
\begin{proof}
Let $x$ be a point for which the second part of Lemma~\ref{oldsta} applies, hence almost every $x$ with $J_u(x)\neq 0$. Let us call for brevity $M$ the matrix $Du(x)$, and let $\delta=\delta(\eps,M)$ be a small constant to be specified later: Lemma~\ref{oldsta} provides us with a constant $\bar r_3(x)>0$ such that the estimates~(\ref{stimagen}) and~(\ref{stimapar}) hold with $\delta$ in place of $\eps$ whenever $r<\bar r_3$. As a consequence, calling $y=u(x)$, and calling ${\rm tr}\, M$ the trace of the matrix $M$, we can calculate
\[\begin{split}
\int_{u(Q_r(\tilde x))} |Du^{-1}(w)|\, dw & \geq
\int_{u(Q_r(\tilde x))} |Du^{-1}(y)|\, dw - \int_{u(Q_r(\tilde x))} |Du^{-1}(w)- Du^{-1}(y)|\, dw\\
&\geq |M^{-1}| \big|u(Q_r(\tilde x))\big| - \delta r^2
\geq |M^{-1}| \big(|\det M| - 2 \delta |{\rm tr}\, M| \big)r^2 - \delta r^2\\
&= |M| r^2 - \delta\big(2 |M^{-1}|\, |{\rm tr}\, M| + 1\big) r^2\,.
\end{split}\]
Since, instead, for $Q_r(\tilde x)$ we have
\[
\int_{Q_r(\tilde x)} |Du(z)|\, dz \leq 
\int_{Q_r(\tilde x)} |Du(x)|\, dz + \int_{Q_r(\tilde x)} |Du(z)-Du(x)|\, dz \leq |M| r^2 + \delta r^2\,,
\]
the validity of~(\ref{flight2}) immediately follows, as soon as $\delta$ has been chosen small enough, again depending only on $\eps$ and on $M$, thus actually on $\eps$ and $x$.
\end{proof}

We can now prove the general validity of~(\ref{1.1}), which can be found already in~\cite[Theorem~1.1]{DDSS}.
\begin{theorem}\label{NRGNRG-1}
Let $u:\Omega\to\Delta$ be a planar, bi-Sobolev homeomorhism. Then,
\begin{equation}\label{1.1}
\int_\Omega |Du| = \int_\Delta |Du^{-1}|\,.
\end{equation}
\end{theorem}
\begin{proof}
Let us subdivide $\Omega$ in four disjoint subsets. Namely, if $x\in \Omega$ is a Lebesgue point of $Du$, then we set $x\in A$ if $|J_u(x)|>0$ and Lemma~\ref{lucio2} holds at $x$, or $x\in B$ if $J_u(x)=0$ but $Du(x)\neq 0$, or $x\in C$ if $Du(x)=0$. Finally, $D\subseteq \Omega$ is the set of points which are not Lebesgue points for $Du$, or such that $J_u(x)\neq 0$ but Lemma~\ref{lucio2} does not hold. Keep in mind that, given $\eps>0$, for every $x\in A\cup B$ we have a constant $\bar r_3(x)$ given by Lemma~\ref{lucio2} or Lemma~\ref{lucio}. For every $r>0$, consider the squares $Q_i$ of the $r$-tiling of $\Omega$ (see Section~\ref{prelim}), and define the set $F_r$ as
\[
F_r = \Big\{ x\in A\cup B:\, \bar r_3(x) \leq r \Big\} \cup \Big(\Omega\setminus \cup\, Q_i\Big) \cup D\,.
\]
Of course, if $r$ is small enough then we can assume $|F_r |< \eps$.\par

For any square $Q_i$ of the $r$-tiling, there are then two possibilities: either $Q_i$ contains at least a point $x\in (A\cup B)\setminus F_r$, and then either Lemma~\ref{lucio} or Lemma~\ref{lucio2} imply
\begin{equation}\label{flight3}
\int_{u(Q_i)} |Du^{-1}| \geq (1-\eps) \int_{Q_i} |Du|\,,
\end{equation}
or otherwise $Q_i$ is entirely contained in $F_r\cup C$. As a consequence, if we call $U$ the union of the squares for which the inequality~(\ref{flight3}) holds true, then $\Omega\setminus U \subseteq F_r\cup C$. Since the different squares are disjoint, and the same is true for their images under the homeomorphism $u$, we can add up the inequalities~(\ref{flight3}) to get
\[\begin{split}
\int_\Delta |Du^{-1}| &\geq 
\int_{u(U)} |Du^{-1}|
\geq (1-\eps) \int_U |Du|
\geq (1-\eps) \int_\Omega |Du| - \int_{\Omega\setminus U} |Du|\\
&\geq (1-\eps) \int_\Omega |Du| - \int_{F_r\cup C} |Du|
= (1-\eps) \int_\Omega |Du| - \int_{F_r} |Du|\,,
\end{split}\]
also recalling the definition of $C$. Since the same argument can be repeated for every $\eps>0$, the measure of $F_r$ is less than $\eps$, and $Du\in L^1(\Omega)$, we derive
\[
\int_\Delta |Du^{-1}| \geq \int_\Omega |Du|\,.
\]
By the symmetry of the situation, we get also the opposite inequality, so~(\ref{1.1}) follows and the proof is concluded.
\end{proof}

As a consequence to Lemma~\ref{lucio}, we can now easily show that $Du=0$ wherever $J_u=0$ for a bi-Sobolev function. This fact is already known, for instance see~\cite[Theorem~4.5]{CHM}, but also~\cite{HK}.
\begin{theorem}[$J_u=0\Longrightarrow Du=0$]\label{sbordone}
Let $u:\Omega\to \Delta$ be a bi-Sobolev homeomorphism. Then, $Du(x)=0$ for almost all the points $x\in\Omega$ such that $J_u(x)=0$.
\end{theorem}
\begin{proof}
Let us call $\Gamma=\big\{x\in\Omega:\, J_u(x)=0,\, Du(x)\neq 0\big\}$, so that we need to prove that $\Gamma$ is a negligible set. For any positive quantity $\eps>0$ we can find $\bar\rho>0$ such that, for any $r<\bar\rho$, the squares $Q_r(x_i)$ of the $r$-tiling of $\Omega$ have a union $U=\cup_{i=1}^N Q_r(x_i)$ which covers the whole $\Omega$ up to a set of area less than $\eps$. Up to take $r$ small enough, anyway smaller than $\bar\rho$, we can also assume that
\[
\Big|\big\{ x \in \Gamma:\, \bar r_3(x) \leq r\big\} \Big| < \eps\,,
\]
where the quantities $\bar r_3(x)$ for points $x\in\Gamma$ are those defined in Lemma~\ref{lucio}. Among the squares $Q_r(x_i)$, let us now pick those which contain some point $x\in \Gamma$ with $\bar r_3(x)> r$, and let us call $V\subseteq U$ their union. If $x\in \Gamma\setminus V$, then either $x\notin U$, or $\bar r_3(x) \leq r$, hence by construction we have
\begin{equation}\label{VconG}
\big|  \Gamma \setminus V\big| < 2\eps\,.
\end{equation}
Let us now consider a square $Q_r(x_i)$ belonging to $V$. This means that there exists some $x\in \Gamma$ with $\bar r_3(x)>r$, and with $x\in Q_r(x_i)$, or equivalently $x_i\in Q_r(x)$. Then, we can apply Lemma~\ref{lucio} and get
\begin{align*}
\big|u(Q_r(x_i))\big| < \eps \int_{u(Q_r(x_i))} |Du^{-1}|\,, && 
\int_{u(Q_r(x_i))} |Du^{-1}| \geq (1-\eps) \int_{Q_r(x_i)} |Du|\,.
\end{align*}
Since the squares are all disjoint, and then the same holds true also for their images under the homeomorphism $u$, summing up and keeping in mind that $u$ is bi-Sobolev, hence $u^{-1}\in W^{1,1}(\Delta)$, we get
\begin{align*}
\big|u(V)\big| < \eps \| Du^{-1}\|_{L^1(\Delta)}\,, && 
\int_{u(V)} |Du^{-1}| \geq (1-\eps) \int_V |Du|\,.
\end{align*}
Recalling~(\ref{VconG}) and sending $\eps$ to $0$, the fact that $Du^{-1}\in L^1(\Delta)$ implies that $\int_\Gamma |Du|=0$, and since $|Du(x)|>0$ for any $x\in \Gamma$ this means that $|\Gamma|=0$, as desired.
\end{proof}

\begin{remark}
It is very simple to observe that the result of Theorem~\ref{sbordone} can be extended to the case of dimension $n\geq 3$. More precisely, Lemma~\ref{oldsta} immediately generalizes to any dimension; the same happens for Lemma~\ref{lucio}, but the assumptions $J_u(x)=0$ and $|Du(x)|>0$ generalize, in dimension $n$, to the assumptions $J_u(x)=0$ and $Rank(Du(x))=n-1$. As a consequence, in dimension $n\geq 3$ the generalization of Theorem~\ref{sbordone} says that for almost every $x\in\Omega$ with $J_u(x)=0$, one has $Rank(Du(x))\leq n-2$. Equivalently, we can say that for almost every $x\in\Omega$ with $J_u(x)=0$, one has $Adj(Du(x))=0$: this result was proved in~\cite[Theorem~4]{HMPS}.
\end{remark}

\section{Approximation of bi-Sobolev homeomorphisms\label{B}}

This section is devoted to show Theorem~\mref{Approxp=1}. Thanks to the general properties of bi-Sobolev homeomorphisms found in Section~\ref{A}, this will be rather simple.

\begin{proof}[Proof of Theorem~\mref{Approxp=1}]
First of all we remark that, as usual, it is enough to show the result for the case of piecewise affine approximating homeomorphisms, because then the case of the diffeomorphisms follows automatically thanks to~\cite{MP}. As a consequence, from now on we look for approximating homeomorphisms which are piecewise affine.\par

Let us fix a small constant $\delta>0$ and let $\Omega^-\subset\subset\Omega$ be a set with finite measure $|\Omega^-|=M$ such that
\begin{equation}\label{pic1}
\int_{\Omega\setminus\Omega^-} |Du| < \delta\,.
\end{equation}
This is really needed only in the case when $\Omega$ has infinite measure, but we can do this in any case, so to give a unique proof without distinguishing the two cases. Let us also fix a small constant $\eps>0$, depending on $M$.\par
For any positive $r>0$, let us then consider the $r$-tiling of $\Omega^-$, which is a subset of the $r$-tiling of $\Omega$, and which is made by finitely many squares. We subdivide then $\Omega^-$ in three parts as follows.\par
The set $\Omega_G$ is the union of all the squares of the $r$-tiling of $\Omega^-$ which contain at least a Lebesgue point $x$ for $Du$ with $J_u(x)\neq 0$, with the property that Lemma~\ref{oldsta} is valid for $x$ (and we know that this is true for almost every $x$ with $J_u(x)\neq 0$), and also that $\bar r_2(x)>r$.\par
The set $\Omega_B$ is the union of all the squares of the $r$-tiling of $\Omega^-$ which are not in $\Omega_G$, and which contain at least a Lebesgue point $x$ for $Du$ with $J_u(x)=0$, with the property that Theorem~\ref{sbordone} is valid for $x$ (and again, this is true for almost every $x$ with $J_u(x)=0$), and also that $\bar r_1(x)>r$.\par
Finally, $\Omega_N=\Omega^-\setminus (\Omega_G\cup\Omega_B)$: thus, $\Omega_N$ is the union of the squares of the $r$-tiling of $\Omega^-$ which are not in $\Omega_G\cup\Omega_B$, together with the portion of $\Omega^-$ which is not covered by the squares of the tiling.\par

Let us consider separately the three sets $\Omega_G,\, \Omega_B$ and $\Omega_N$. First of all we notice that, except for a negligible set, $\Omega_N$ is entirely done by points of $\Omega^-$ which are not in the squares of the tiling, and by Lebesgue points $x$ for $Du$ for which either $\bar r_2(x)$ or $\bar r_1(x)$ is smaller than $r$; as a consequence, as already noticed several times, the measure of $\Omega_N$ is as small as we wish up to take $r$ small enough. Therefore, we can assume that
\begin{equation}\label{pic2}
\int_{\Omega_N} |Du|\leq \eps\,.
\end{equation}\par
We now concentrate ourselves on $\Omega_B$. Let $Q$ be a square of the $r$-tiling of $\Omega'$ contained in $\Omega_B$, and let $x\in Q$ be a Lebesgue point for $Du$ as in the definition of $\Omega_B$. Thus, $J_u(x)=0$ and, thanks to Theorem~\ref{sbordone}, it is also $Du(x)=0$. Since $r<\bar r_1(x)$, the estimate~(\ref{stimagen}) of Lemma~\ref{oldsta} gives
\[
\int_Q |Du| \leq \int_Q |Du(z) - Du(x) | \, dz + \int_Q |Du(x)|\, dz = \int_Q |Du(z) - Du(x) | \, dz \leq \eps r^2 = \eps |Q|\,.
\]
Summing up on all the squares contained in $\Omega_B$, we find then
\begin{equation}\label{pic3}
\int_{\Omega_B} |Du| \leq \eps |\Omega_B| \leq \eps M\,.
\end{equation}\par
Now, we pass to consider $\Omega_G$. First of all, putting together~(\ref{pic1}), (\ref{pic2}) and~(\ref{pic3}) we already know that
\begin{equation}\label{OGgr}
\int_{\Omega\setminus \Omega_G} |Du| =
\int_{\Omega\setminus \Omega^-} |Du|+ \int_{\Omega_N} |Du|+ \int_{\Omega_B} |Du|
\leq \delta + \eps(1+M)\,.
\end{equation}
Let now $Q$ be a generic square of $\Omega_G$: Corollary~\ref{bafo} gives
\begin{equation}\label{for_G}
\int_Q |Du-Du_Q| + \int_{u_Q(Q)} |Du^{-1} - Du_Q^{-1}| < 5 \eps r^2=5\eps |Q|\,.
\end{equation}
Let us now apply Proposition~\ref{standa} with the squares of $\Omega_G$. Notice that this is possible: indeed, they are squares of the $r$-tiling of $\Omega^-$, thus also of the $r$-tiling of $\Omega$, and they are finitely many because the $r$-tiling of $\Omega^-$ is finite. Moreover, each of these squares contains a Lebesgue point $x$ for $Du$ with $J_u(x)\neq 0$ and $\bar r_2(x)>r$, then a fortiori $\bar r_1(x)> r$. Proposition~\ref{standa} provides us then with a piecewise affine homeomorphism $v$ on $\Omega$ satisfying the conditions~(i)--(v). All the requirements of Theorem~\mref{Approxp=1}, except the estimate~(\ref{estimate}), are then satisfied by $u_\eta=v$; thus, we only have to show also the validity of~(\ref{estimate}).\par

Since for every square $Q$ of $\Omega_G$ the function $v$ coincides with $u_Q$ in $Q$, adding the estimates~(\ref{for_G}) gives
\begin{equation}\label{partefacile}
\int_{\Omega_G} |Du-Dv| + \int_{v(\Omega_G)} |Du^{-1} - Dv^{-1}| < 5\eps |\Omega_G|\leq 5 \eps M\,.
\end{equation}
As a consequence, we get
\begin{equation}\label{prestuff}
\bigg|\int_{v(\Omega_G)} |Du^{-1}| - \int_{\Omega_G} |Du| \bigg|
\leq \bigg|\int_{v(\Omega_G)} |Dv^{-1}| - \int_{\Omega_G} |Dv|\bigg| + 5\eps M
=5\eps M\,.
\end{equation}
Here we have used the fact that $v$ is a finitely piecewise affine homeomorphism between $\Omega_G$ and $v(\Omega_G)$, and then we can apply Theorem~\ref{NRGNRG-1} to $v$ on $\Omega_G$. We claim that it is also
\begin{equation}\label{senzath}
\int_{\Delta\setminus v(\Omega_G)} |Dv^{-1}| = \int_{\Omega\setminus \Omega_G} |Dv|\,.
\end{equation}
Here, we cannot use Theorem~\ref{NRGNRG-1}, because still we do not know that $v$ is a bi-Sobolev homeomorphism on the whole $\Omega$. However, as observed in the introduction, the validity of~(\ref{1.1}) is trivially true for an affine map (or, more in general, for a map which satisfies both the $N$ and the $N^{-1}$ Lusin properties). Then, since $v$ is piecewise affine, we have that $\Omega\setminus \Omega_G$ is a finite or countable union of triangles, on each of which $v$ is affine: adding the equality~(\ref{1.1}) for $v$ on each of these triangles, we obtain the validity of~(\ref{senzath}). Therefore, (\ref{senzath}), property~(v) of Proposition~\ref{standa}, Theorem~\ref{NRGNRG-1}, (\ref{prestuff}) and~(\ref{OGgr}) give
\[\begin{split}
\int_{\Omega\setminus \Omega_G} |Du &- Dv| + \int_{\Delta\setminus v(\Omega_G)} |Du^{-1} - Dv^{-1}|\\
&\leq \int_{\Omega\setminus \Omega_G} |Du| +\int_{\Omega\setminus \Omega_G} |Dv|
+\int_{\Delta\setminus v(\Omega_G)} |Du^{-1}| +\int_{\Delta\setminus v(\Omega_G)} |Dv^{-1}|\\
&\leq (2K+1)\int_{\Omega\setminus \Omega_G} |Du| +\int_\Delta |Du^{-1}| - \int_{v(\Omega_G)} |Du^{-1}|\\
&= (2K+1)\int_{\Omega\setminus \Omega_G} |Du|+ \int_\Omega |Du| - \int_{v(\Omega_G)} |Du^{-1}|\\
&\leq (2K+2)\int_{\Omega\setminus \Omega_G} |Du|+ 5 \eps M
\leq (2K+2) \big(\delta + \eps(1+M)\big) + 5 \eps M\,.
\end{split}\]
Recall that $\delta$ is an arbitrary constant, $M$ depends on $\delta$, and $\eps$ can be chosen in dependence on $M$, while $K$ is a purely geometric constant. As a consequence, putting together the last estimate, (\ref{partefacile}), and condition~(i) of Proposition~\ref{standa}, we get the validity of~(\ref{estimate}) up to chose $\delta$ and $\eps$ small enough with respect to $\eta$. The proof is then concluded.
\end{proof}

\subsection*{Acknowledgment}
The author wishes to thank S. Hencl, for his precious help through the literature, and C. Sbordone, for many useful conversations about bi-Sobolev mappings. This research was partially supported by the ERC St.G. ``AnOptSetCon'' of the European Community.

\end{document}

%% file: Jf=0.tex
\begin{picture}(0,0)%
\includegraphics{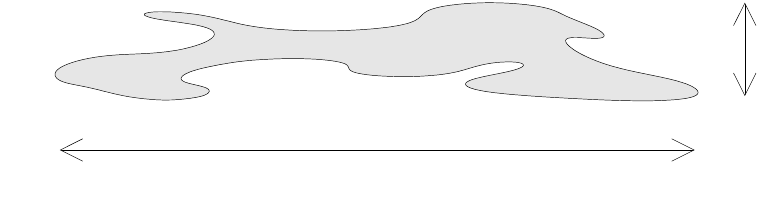}%
\end{picture}%
\setlength{\unitlength}{1184sp}%
\begingroup\makeatletter\ifx\SetFigFont\undefined%
\gdef\SetFigFont#1#2#3#4#5{%
  \reset@font\fontsize{#1}{#2pt}%
  \fontfamily{#3}\fontseries{#4}\fontshape{#5}%
  \selectfont}%
\fi\endgroup%
\begin{picture}(12230,3137)(-7914,-6129)
\put(-2628,-5961){\makebox(0,0)[lb]{\smash{{\SetFigFont{10}{12.0}{\rmdefault}{\mddefault}{\updefault}$\approx Lr$}}}}
\put(4301,-3861){\makebox(0,0)[lb]{\smash{{\SetFigFont{10}{12.0}{\rmdefault}{\mddefault}{\updefault}$\leq\delta r$}}}}
\put(-7899,-3461){\makebox(0,0)[lb]{\smash{{\SetFigFont{10}{12.0}{\rmdefault}{\mddefault}{\updefault}$u(Q_r(\tilde x))$}}}}
\end{picture}%

%% file: biSobolev.bbl
\begin{thebibliography}{999}
\bibitem{Ball2} J.M. Ball, Progress and puzzles in Nonlinear Elasticity, Poly-, Quasi- and Rank-One Convexity in Applied Mechanics, CISM International Centre for Mechanical Sciences Vol. {\bf 516} (2010), 1--15.
\bibitem{CHM} M. Cs\"ornyei, S. Hencl \& J. Mal\'y, Homeomorphisms in the Sobolev space $W^{1,n-1}$, J. Reine Angew. Math. {\bf 644} (2010), 221--235.
\bibitem{DP} S. Daneri \& A. Pratelli, Smooth approximation of bi-Lipschitz orientation-preserving homeomorphisms, Ann. Inst. H. Poincar\'e Anal. Non Lin\'eaire, {\bf 31}, n. 3 (2014), 567--589\,.
\bibitem{DHS} L. D'Onofrio, S. Hencl \& R. Schiattarella, Bi-Sobolev homeomorphism with zero Jacobian almost everywhere, Calc. Var. PDE {\bf 51} (2014), 139--170.
\bibitem{DDSS} P. Di Gironimo, L. D'Onofrio, C. Sbordone \& R. Schiattarella, Anisotropic Sobolev homeomorphisms, Ann. Acad. Sci. Fenn. Math. {\bf 36} (2011), no. 2, 593--602.
\bibitem{HK} S. Hencl \& P. Koskela, Regularity of the Inverse of a Planar Sobolev Homeomorphism, Arch. Rational Mech. Anal. {\bf 180} (2006), 75--95.
\bibitem{HKLectureNotes} S. Hencl \& P. Koskela, Lectures on Mappings of finite distorsion, Lecture Notes in Mathematics {\bf 2096}, Springer (2014).
\bibitem{HMPS} S. Hencl, G. Moscariello, A. Passarelli di Napoli \& C. Sbordone, Bi-Sobolev mappings and elliptic equations in the plane, J. Math. Anal. Appl. {\bf 355} (2009), no. 1, 22--32.
\bibitem{HP} S. Hencl \& A. Pratelli, Diffeomorphic Approximation of $W^{1,1}$ Planar Sobolev Homeomorphisms, preprint (2014). Available at http://arxiv.org/abs/1502.07253\,.
\bibitem{IwKovOnnold} T. Iwaniec, L. V. Kovalev and J. Onninen, Hopf differentials and smoothing Sobolev homeomorphisms, Int. Math. Res. Not. IMRN 2012 (2012), no. 14, 3256--3277.
\bibitem{IwKovOnn} T. Iwaniec, L. V. Kovalev \& J. Onninen, Diffeomorphic Approximation of Sobolev Homeomorphisms, Arch. Rational Mech. Anal. {\bf 201} (2011), no. 3, 1047--1067.
\bibitem{MP} C. Mora Corral \& A. Pratelli, Approximation of piecewise affine homeomorphisms by diffeomorphisms, J. Geom. Anal. {\bf 24} (2014), no. 3, 1398--1424.
\bibitem{Pon} S.P. Ponomarev, Property $N$ of homeomorphisms of the class $W^{1,p}$, Siberian Mathematical Journal {\bf 28}, n. 2 (1987), 291--298.
\end{thebibliography}
